\documentclass[11pt,reqno,a4paper]{amsart}
\usepackage{euscript}

\newtheorem{thm}{Theorem}
\newtheorem{lm}{Lemma}

\newtheorem{prop}{Proposition}

\renewcommand{\epsilon}{\varepsilon}
\renewcommand{\phi}{\varphi}

\def\Z{\mathbb{Z}}
\def\cA{\EuScript{A}}

\def\Id{\text{\rm Id}}

\begin{document}

\title[Linearization of Nonautonomous Difference Equations]{Smooth Linearization of Nonautonomous Difference Equations with a Nonuniform Dichotomy
}


\maketitle

\author{Davor Dragi\v cevi\'c\,$^a$, \ \ \ Weinian Zhang\,$^b$,\ \ \  Wenmeng Zhang\,$^c$
\\   \\
$^a${\small School of Mathematics and Statistics, University of New South Wales}
\\
\indent
{\small Sydney NSW 2052, Australia, d.dragicevic@unsw.edu.au}
\\
$^b${\small School of Mathematics, Sichuan University}
\\
\indent
{\small Chengdu, Sichuan 610064, China, matzwn@126.com}
\\
$^c${\small School of Mathematical Sciences, Chongqing Normal University}
\\
\indent
{\small Chongqing, 401331, China, wmzhang@cqnu.edu.cn}
}

\begin{abstract}
In this paper we give a smooth linearization theorem for nonautonomous difference equations with a nonuniform
strong exponential dichotomy. The linear part of such a nonautonomous difference equation is defined
by a sequence of invertible linear operators on $\mathbb{R}^d$.
Reducing the linear part to a bounded linear operator on a Banach space,
we discuss the spectrum and its spectral gaps.
Then we obtain a gap condition for $C^1$ linearization of such a nonautonomous difference equation.
.Our theorems improve known results even in the case of uniform strong  exponential dichotomies.
\end{abstract}


{\footnotesize \sc MSN (2010):}
37C60;  
37D25.  

\parskip 0.3cm


\section{Introduction}
\label{sec.1}

Linearization is one of the most fundamental and important problems in the theory of dynamical systems.
It answers whether a dynamical system is conjugated to its linear part in the sense of $C^r$ ($r\ge 0$).
Linearization is a powerful tool in the discussion of qualitative properties. One of earliest works
was given by Poincar\'e (\cite{Poin}), who proved that an analytic diffeomorphism can be analytically conjugated to
its linear part near a fixed point if all eigenvalues of the linear part lie inside
the unit circle $S^1$ (or outside $S^1$) and satisfy the nonresonant condition. Later, Siegel (\cite{Sieg}), Brjuno (\cite{Brj})
and Yoccoz (\cite{Yoc}) made contributions to the case of eigenvalues on $S^1$, in which the small divisor problem is involved.
The most well-known result is the Hartman-Grobman Theorem (\cite{Hart}), which says that
$C^1$ diffeomorphisms in $\mathbb{R}^n$ can be $C^0$ linearized near the hyperbolic fixed points. Later this result was generalized to
Banach spaces by  Palis (\cite{Palis}) and Pugh (\cite{Pugh}).

Sometimes $C^0$ linearization is not effective to discuss more details of dynamics, for example, to distinguish a node from a focus and
to straighten invariant manifolds, which requests results on smooth linearization.
In 1950's Sternberg (\cite{Stern57,Stern58}) proved that $C^k$ ($k\ge1$) diffeomorphisms can be $C^r$ linearized
near the hyperbolic fixed points, where the integer $r$ depends on $k$ and the nonresonant condition.
In 1970's Belitskii (\cite{Bel73,Bel78}) gave conditions on $C^k$ linearization for $C^{k,1}$ ($k\ge 1$)
diffeomorphisms, which implies that $C^{1,1}$ diffeomorphisms can be $C^1$ linearized
locally if the eigenvalues $\lambda_1,...,\lambda_n$ satisfy that $|\lambda_i|\cdot|\lambda_j|\ne |\lambda_\iota|$
for all $\iota=1,...,n$ if $|\lambda_i|<1<|\lambda_j|$.
As we know for structural stability,
$C^1$ smoothness of conjugacies is of special significance in distinguishing various dynamical systems.
Hartman (\cite{Hart60}) proved that all $C^{1,1}$ contractions on $\mathbb{R}^n$ admit local $C^{1,\beta}$
linearization with small $\beta>0$.
In the early new millennium, conditions on $C^1$ linearization were obtained in Banach spaces
in \cite[Corollary 1.3.3]{Chap} and \cite{ElB, R-S-JDE04, R-S-JDDE04}.
Recently,
weaker conditions for $C^1$ linearization were obtained in \cite{ZhangZhang11JFA, ZhangZhang14JDE} in $\mathbb{R}^2$
and in \cite{ZZJ} in Banach spaces.
Differentiable depedence on parameters was discussed in \cite{R-S-JDE17}.

The above investigation on diffeomorphisms can be regarded as the problem of linearization for autonomous difference equations.
In this paper we aim to nonautonomous difference equations, a more general case.
A general first order nonautonomous difference equation having a fixed point at the origin $O$
is of the form
$$
x_{m+1}=A_m x_m+ f_m(x_m),
$$
which is defined by a sequence $(A_m)_{m\in \mathbb{Z}}$ of invertible linear operators on $\mathbb{R}^d$ and
a sequence $(f_m)_{m\in \mathbb{Z}}$ of smooth functions $f_m:\mathbb{R}^d\to \mathbb{R}^d$ such that $f_m(0)=0$.
Its linearization means the existence of a sequence $(h_m)_{m\in\Z}$ of homeomorphisms near $O$ on $\mathbb{R}^d$
such that
$$
h_{m+1}\circ (A_m+ f_m)=A_m\circ h_m, \ \ \  \forall m\in \Z.
$$
The first nonautonomous version of the Hartman-Grobman Theorem is due to Palmer~\cite{Palmer}
but for the differential equation $x'=A(t)x$ on $\mathbb R^d$ and under the assumption of
(uniform) exponential dichotomy.
In \cite{BV1} Barreira and Valls discussed H\"older continuous
linearization for nonautonomous difference equations
with nonuniform dichotomy, but nothing deals with the smoothness.

In this paper we give a smooth linearization theorem for nonautonomous difference equations with a nonuniform
strong exponential dichotomy. The linear part of such a nonautonomous difference equation is defined
by a sequence of invertible linear operators on $\mathbb{R}^d$.
As in~\cite{BDV1}, we construct a bounded linear operator on a Banach space to convert
the linear part in the nonautonomous setting to an autonomous one,
so that we can discuss the spectrum and its spectral gaps conveniently.
Then we obtain a gap condition for $C^1$ linearization of such a nonautonomous difference equation.
We finally extend the result to compact operators in the infinite dimensional case.

It is worthy mentioning that, in addition to the nonautonomous form of systems, 
another contribution of this paper is the nonuniform version of dichotomies.
Although there were published some results (\cite{BarVal-DCDS07, WL-JDE08}) with uniform exponential dichotomies,
to the best of our knowledge, all the results of this paper are new even
in the particular case that the sequence $(A_m)_{m\in \Z}$ admits a uniform exponential dichotomy
and satisfies the condition that
$\sup_{m\in \Z} \max \{\lVert A_m\rVert, \lVert A_m^{-1}\rVert \} <\infty$.
The principal motivation for considering the notion of a nonuniform strong exponential dichotomy 
comes from its ubiquity in the context of ergodic theory. 
Indeed,
let $(\Omega, \mathcal F, \mathbb P)$ be a probability space and assume that $\sigma \colon \Omega \to \Omega$ is a measurable and invertible transformation that preserves
$\mathbb P$. We recall that this means that $\mathbb P(A)=\mathbb P(\sigma^{-1}(A))$ for each $A\in \mathcal F$. Finally, suppose that $\mathbb P$ is ergodic, i.e. that
$\mathbb P(A)\in \{0, 1\}$ for every $A\in \mathcal F$ such that  $\sigma^{-1}(A)=A$. Let $GL_d$ denote the space of all regular matrices of order d and consider a measurable map
(the so-called linear cocycle)
$A\colon \Omega \to GL_d$ such that
\[
 \int_{\Omega} \max \{0, \log \lVert A(\omega)\rVert\} \, d\mathbb P(\omega) <\infty \ \text{and} \ \int_{\Omega} \max \{0, \log \lVert A(\omega)^{-1}\rVert\} \, d\mathbb P(\omega) <\infty
.\]
Then, if  all Lyapunov exponent of $A$ with respect to $\mathbb P$ 
(given by the Oseledets multiplicative ergodic theorem~\cite{Osel}) are nonzero, we have that for $\mathbb P$-a.e. 
$\omega \in \Omega$, the sequence $(A_m)_{m\in \Z}$ defined by $A_m=A(\sigma^m(\omega))$, $m\in \Z$, 
admits a strong nonuniform exponential dichotomy (see~\cite{BVbook}). We refer
to~\cite{BP} for a detailed exposition of the theory of the dynamical systems with nonzero Lyapunov exponents, 
which goes back to the pioneering work of Pesin~\cite{Pesin}.


\section{Nonuniform strong  exponential dichotomies}
\label{sec.2}

Let $(A_m)_{m \in \Z}$ be a (two-sided) sequence of invertible  linear operators on $\mathbb R^d$.
It defines the nonautonomous difference equation
\begin{eqnarray}
x_{n+1}=A_nx_n.
\label{NA}
\end{eqnarray}
For each $m,n \in \Z$ we define
\begin{eqnarray}
\cA(m,n)=\begin{cases}
A_{m-1} \cdots A_n & \text{if $m > n$,} \\
\Id & \text{if $m=n$,}\\
A_m^{-1} \cdots A_{n-1}^{-1} & \text{if $m<n$.}
\end{cases}
\label{AAA}
\end{eqnarray}
Then $\cA(m,n)$ is a fundamental solution operator of (\ref{NA}) because for any initial point $x_0$ the solution of (\ref{NA})
can be presented as $x_m=\cA(m,0)x_0$.
We say that the sequence $(A_m)_{m \in \Z}$ has a \emph{ nonuniform strong  exponential dichotomy}
if there exist projections $P_m$, $m \in \Z$,
satisfying
\begin{equation}\label{proj}
\cA(m,n)P_n=P_m\cA(m,n)
\ \ \
\forall m,n \in \Z
\end{equation}
and there exist constants
\begin{equation}\label{*EE}
0<\lambda \le \mu , \quad \epsilon \ge 0 \quad \text{and} \quad D >0
\end{equation}
such that for $m \ge n$
\begin{eqnarray*}
\begin{array}{ll}
\lVert \cA(m,n)P_n \rVert \le De^{-\lambda(m-n)+\epsilon |n|}, &\lVert \cA(m,n)Q_n \rVert \le De^{\mu(m-n) +\epsilon |n|},
\\
\lVert \cA(n,m)Q_m \rVert \le De^{-\lambda (m-n)+\epsilon |m|}, &\lVert \cA(n,m)P_m \rVert \le De^{\mu (m-n)+\epsilon |m|},
\end{array}
\end{eqnarray*}
where $Q_m:=\Id-P_m$ for $m\in \Z$.
Furthermore,
we say that the sequence $(A_m)_{m\in \Z}$ admits a \emph{strong exponential dichotomy
with respect to the sequence of norms}
$(\lVert \cdot \rVert_m)_{m\in \Z}$, each of which is a norm on $\mathbb R^d$,
if there exist projections $P_m$, $m\in \Z$, satisfying~\eqref{proj} and there exist constants as given in~\eqref{*EE}
with $\epsilon=0$
such that
for $m\ge n$ and $x\in \mathbb R^d$
\begin{eqnarray}
\begin{array}{ll}
\lVert \cA(m,n)P_nx\rVert_m \le De^{-\lambda(m-n)}\lVert x\rVert_n,   &\lVert \cA(m,n)Q_n x\rVert_m \le De^{\mu(m-n) }\lVert x\rVert_n,
\\
\lVert \cA(n,m)Q_m x\rVert_n \le De^{-\lambda (m-n)}\lVert x\rVert_m, &\lVert \cA(n,m)P_mx \rVert_n \le De^{\mu (m-n)}\lVert x\rVert_m,
\end{array}
\label{fds}
\end{eqnarray}
where $Q_m:=\Id-P_m$ for $m\in \Z$.
The relationship between those two concepts of dichotomy is given by the following result.

\begin{lm}{\bf (Proposition 2.2 in \cite{BDV})}
\label{prop1}
 Let $(A_m)_{m\in \Z}$ be a sequence of invertible operators on $\mathbb R^d$. The following properties are equivalent:
 \begin{enumerate}
  \item $(A_m)_{m\in \Z}$ admits a  nonuniform strong  exponential dichotomy;
  \item $(A_m)_{m\in \Z}$ admits a strong exponential dichotomy with respect to a sequence of norms
$(\lVert \cdot \rVert_m)_{m\in \Z}$ with the property that there exist $C>0$ and $\epsilon \ge 0$ such that
  \begin{equation}\label{ln1}
   \lVert x\rVert \le \lVert x\rVert_m \le Ce^{\epsilon \lvert m\rvert} \lVert x\rVert
  \end{equation}
  and
  \begin{equation}\label{ln2}
   \frac 1C \lVert x\rVert_m \le \lVert A_m x\rVert_{m+1}\le C\lVert x\rVert_m,
  \end{equation}
for every $x\in \mathbb R^d$ and $m\in \Z$.
 \end{enumerate}
\end{lm}

For a sequence of norms $(\lVert \cdot \rVert_m)_{m\in \Z}$  on $\mathbb R^d$, let
\[
 Y_\infty=\bigg{\{} \mathbf x=(x_n)_{n\in \Z} \subset \mathbb R^d: \sup_{n\in \Z} \lVert x_n\rVert_n <+\infty \bigg{\}},
\]
which is a Banach space equipped with the norm
\[
 \lVert \mathbf x\rVert_\infty=\sup_{n\in \Z} \lVert x_n\rVert_n.
\]
Assume that~\eqref{ln2} holds and define a bounded linear operator $\mathbb A \colon Y_\infty \to Y_\infty$ by
\begin{equation}\label{eq:A}
(\mathbb A \mathbf x)_n=A_{n-1}x_{n-1}, \quad \mathbf x=(x_n)_{n\in \Z} \in Y_\infty, \ n \in \Z.
\end{equation}
One can easily verify that
$\mathbb A$ is invertible. Indeed, using the first  inequality in~\eqref{ln2}, we find that the inverse of
$\mathbb A$ is the operator $\mathbb B$ given by
\[
(\mathbb B \mathbf x)_n=A_n^{-1}x_{n+1}, \quad \mathbf x=(x_n)_{n\in \Z} \in Y_\infty, \ n \in \Z.
\]
We will also need the following result.

\begin{lm}{\bf (Theorem 6.1 in \cite{BDV})}
\label{t1}
 Let $(A_m)_{m\in \Z}$ be a sequence of invertible linear operators on $\mathbb R^d$ and let $(\lVert \cdot \rVert_m)_{m\in \Z}$
be a sequence of norms on $\mathbb R^d$ such that~\eqref{ln2} holds
 with some $C>0$. Then the following statements are equivalent:
\begin{enumerate}
\item[\bf (i)]
$(A_m)_{m\in \Z}$ admits a strong exponential dichotomy with respect to  the
sequence of norms $(\lVert \cdot \rVert_m)_{m\in \Z}$;

\item[\bf (ii)]
the operator $\Id-\mathbb A$ is invertible on $Y_\infty$, where $\mathbb A$ is given by~\eqref{eq:A}.
\end{enumerate}
\end{lm}


\section{Spectrum of $\mathbb A$}

Assume that $(A_m)_{m\in \Z}$ admits a nonuniform strong exponential dichotomy.
Throughout the following two sections, choose
$(\lVert \cdot \rVert_m)_{m\in \Z}$ to be the
sequence of norms given by Lemma~\ref{prop1}.
Furthermore, consider $\mathbb A$ defined as in~\eqref{eq:A}. Let $\sigma (\mathbb A)$ denote the spectrum of $\mathbb A$, i.e.,
\[
 \sigma (\mathbb A)=\{ a\in \mathbb C: a\Id-\mathbb A \ \text{is not invertible on $Y_\infty$}\}.
\]
The following result is a direct consequence of Lemma~\ref{t1}.

\begin{prop}\label{p2x}
Let $b\in \mathbb C$ such that $\lvert b\rvert=1$ and $a\in \sigma (\mathbb A)$. Then, $ab\in \sigma (\mathbb A)$.
\end{prop}

\begin{proof}
 Assume that $ab\notin \sigma (\mathbb A)$. Hence, $ab\Id-\mathbb A$ is an invertible operator on $Y_\infty$.
Hence, it follows from Theorem~\ref{t1} that the sequence $(\frac{1}{ab}A_m)_{m\in \Z}$
admits a strong exponential dichotomy with respect to norms $\lVert \cdot \rVert_m$, $m\in \Z$. However, since $\lvert b\rvert=1$ this would imply that
$(\frac{1}{a}A_m)_{m\in \Z}$ admits a strong exponential dichotomy with respect to norms $\lVert \cdot \rVert_m$, $m\in \Z$.
Using again Theorem~\ref{t1}, we would obtain that
$a\Id-\mathbb A$ is invertible on $Y_\infty$ which yields a contradiction with our assumption that $a\in \sigma (\mathbb A)$.
\end{proof}

The proof of the following result is inspired by the classical work of Sacker and Sell (\cite{SS}).

\begin{prop}\label{p3x}
 $\sigma (\mathbb A)\cap (0, \infty)$ is a  union of finitely many disjoint closed intervals.
\end{prop}

\begin{proof}
 For arbitrary $a>0$ and $n\in \Z$, let
 \[
  S_a(n):=\bigg{\{}x\in \mathbb R^d: \sup_{m\ge n}\frac{1}{a^{m-n}}\lVert \cA(m, n)x\rVert_m <\infty \bigg{\}}.
 \]
If $a\notin \sigma (\mathbb A)$, then $(\frac 1 aA_m)_{m\in \Z}$ admits a strong exponential dichotomy
with respect to the
sequence of norms $(\lVert \cdot \rVert_m)_{m\in \Z}$ and projections
$P_m$, $m\in \Z$.
It is easy to verify that $\mathcal{R}P_m$, the range of $P_m$, satisfies
$\mathcal{R}P_m=S_a(m)$
for
all $m\in \Z$.
Hence, $\dim S_a(m)$ does not depend on $m$ and we can write simply $\dim S_a$.

{\bf Claim 1}
{\it For any $a\in (0, \infty)\setminus \sigma (\mathbb A)$, there exists $\rho >0$ such that for each $b\in (a-\rho, a+\rho)$, we have that
 $b\notin \sigma (\mathbb A)$ and $\dim S_a=\dim S_b$.
}


In fact, since $a\in (0, \infty)\setminus \sigma (\mathbb A)$, the sequence $(\frac 1 aA_m)_{m\in \Z}$
admits a strong exponential dichotomy with respect to the
sequence of norms $(\lVert \cdot \rVert_m)_{m\in \Z}$.
Thus, there exist projections $P_m$, $m\in \Z$, satisfying~\eqref{proj} and constants as given in~\eqref{*EE}
with $\epsilon=0$
such that for $m\ge n$ and $x\in \mathbb R^d$
\[
\frac{1}{a^{m-n}}\lVert \cA(m,n)P_nx\rVert_m \le De^{-\lambda(m-n)}\lVert x\rVert_n,
\]
\[
\frac{1}{a^{m-n}}\lVert \cA(m,n)Q_n x\rVert_m \le De^{\mu(m-n) }\lVert x\rVert_n,
\]
\[
 \frac{1}{a^{n-m}}\lVert \cA(n,m)Q_m x\rVert_n \le De^{-\lambda (m-n)}\lVert x\rVert_m,
\]
and
\[
  \frac{1}{a^{n-m}}\lVert \cA(n,m)P_mx \rVert_n \le De^{\mu (m-n)}\lVert x\rVert_m.
\]
Choose $\rho >0$ such that
\[
 \frac{a}{a-\rho}e^{-\lambda} <1 \quad \text{and} \quad \frac{a+\rho}{a}e^{-\lambda} <1.
\]
Then, for each $b\in (a-\rho, a+\rho)$, we have
\[
\frac{1}{b^{m-n}}\lVert \cA(m,n)P_n x\rVert_m \le D\bigg{(} \frac{a}{a-\rho}e^{-\lambda} \bigg{)}^{m-n} \lVert x\rVert_n,
\]
\[
\frac{1}{b^{m-n}}\lVert \cA(m,n)Q_n x\rVert_m \le D\bigg{(}  \frac{a}{a-\rho}e^{\mu} \bigg{)}^{m-n} \lVert x\rVert_n,
\]
\[
 \frac{1}{b^{n-m}}\lVert \cA(n,m)Q_m x\rVert_n \le D \bigg{(} \frac{a+\rho}{a}e^{-\lambda} \bigg{)}^{m-n}\lVert x\rVert_m,
\]
and
\[
  \frac{1}{a^{n-m}}\lVert \cA(n,m)P_mx \rVert_n \le D \bigg{(} \frac{a+\rho}{a}e^{\mu} \bigg{)}^{m-n}\lVert x\rVert_m,
\]
for $m\ge n$ and $x\in \mathbb R^d$. This implies that $(\frac 1 bA_m)_{m\in \Z}$ admits
a strong exponential dichotomy with respect to the
sequence of norms $(\lVert \cdot \rVert_m)_{m\in \Z}$
and projections $P_m$, $m\in \Z$. Using Lemma~\ref{t1}, we obtain  that $b\notin \sigma (\mathbb A)$. In addition,
\[
 S_b(m)=\mathcal R P_m =S_a(m) \quad \text{for each $m\in \Z$,}
\]
which implies that $\dim S_a=\dim S_b$ and proves the claim.

We note that for $0<a_1<a_2$, we have that $S_{a_1}(m) \subset S_{a_2}(m)$ for each $m\in \Z$.
We further need the following.

{\bf Claim 2}
{\it
Assume that $0<a_1<a_2$ and $a_1, a_2 \notin \sigma (\mathbb A)$. The following two statements are equivalent:
{\bf (i)}
$\dim S_{a_1}=\dim S_{a_2}$;
{\bf (ii)} $[a_1, a_2] \cap \sigma (\mathbb A)=\emptyset$.
}

In fact, assume that $\dim S_{a_1}=\dim S_{a_2}$. Hence, sequences $(\frac{1}{a_i}A_m)_{m\in \Z}$, $i=1, 2$,
admit the same
strong exponential dichotomy with respect to the same
sequence of norms
$(\lVert \cdot \rVert_m)_{m\in \Z}$  and the same projections $P_m$, $m\in \Z$.
Then, there exist constants as given in~\eqref{*EE} with $\epsilon=0$
such that for $m\ge n$, $x\in \mathbb R^d$ and $i=1, 2$
\[
\frac{1}{a_i^{m-n}}\lVert \cA(m,n)P_nx\rVert_m \le De^{-\lambda(m-n)}\lVert x\rVert_n,
\]
\[
\frac{1}{a_i^{m-n}}\lVert \cA(m,n)Q_n x\rVert_m \le De^{\mu(m-n) }\lVert x\rVert_n,
\]
\[
 \frac{1}{a_i^{n-m}}\lVert \cA(n,m)Q_m x\rVert_n \le De^{-\lambda (m-n)}\lVert x\rVert_m,
\]
and
\[
 \frac{1}{a_i^{n-m}}\lVert \cA(n,m)P_mx \rVert_n \le De^{\mu (m-n)}\lVert x\rVert_m.
\]
Then, for $a\in [a_1, a_2]$ we have
\[
  \frac{1}{a^{m-n}}\lVert \cA(m,n)P_n x\rVert_m\le  \frac{1}{a_1^{m-n}}\lVert \cA(m,n)P_n x\rVert_m \le De^{-\lambda(m-n)}\lVert x\rVert_n,
\]
\[
  \frac{1}{a^{m-n}}\lVert \cA(m,n)Q_n x\rVert_m \le   \frac{1}{a_1^{m-n}}\lVert \cA(m,n)Q_n x\rVert_m \le De^{\mu(m-n) }\lVert x\rVert_n,
\]
\[
 \frac{1}{a^{n-m}}\lVert \cA(n,m)Q_m x\rVert_n \le \frac{1}{a_2^{n-m}}\lVert \cA(n,m)Q_m x\rVert_n \le De^{-\lambda (m-n)}\lVert x\rVert_m,
\]
and
\[
 \frac{1}{a^{n-m}}\lVert \cA(n,m)Q_m x\rVert_n \le  \frac{1}{a_2^{n-m}}\lVert \cA(n,m)P_mx \rVert_n \le De^{\mu (m-n)}\lVert x\rVert_m,
\]
for $m\ge n$ and $x\in \mathbb R^d$. We conclude that $(\frac 1 a A_m)_{m\in \Z}$ admits a strong exponential dichotomy
with respect to
the sequence of norm $(\lVert \cdot \rVert_m)_{m\in \Z}$
and thus $a\notin \sigma (\mathbb A)$.
Let us now assume that $[a_1, a_2] \cap \sigma (\mathbb A)=\emptyset$ and suppose that $\dim S_{a_1}< \dim S_{a_2}$.
Let
\[
 c:=\inf \{b\notin \sigma (\mathbb A): \dim S_b=\dim S_{a_2} \}.
\]
By Claim 1,
it is easy to conclude that $c\in \sigma (\mathbb A) \cap (a_1, a_2)$ which yields a contradiction.
This proves the claim.

Then, the conclusion of our Proposition follows easily Claim 2.
\end{proof}

\begin{thm}\label{ts}
Let the linear operator $\mathbb A$
be defined as in (\ref{eq:A}) by a sequence of invertible linear operators $(A_m)_{m\in \Z}$ on $\mathbb R^d$
with a nonuniform strong exponential dichotomy.
Then there exist real constants
 \[
  0<a_1 \le b_1 <a_2 \le b_2 <\ldots <a_k \le b_k <1 < a_{k+1} \le b_{k+1} < \ldots <a_r \le b_r
 \]
such that
\begin{equation}\label{spec-cond}
 \sigma (\mathbb A)=\bigcup_{i=1}^r \{ z\in \mathbb C: a_i \le \lvert z\rvert \le b_i \}.
\end{equation}
\end{thm}

\begin{proof}
By Lemma~\ref{t1} and Proposition~\ref{p2x}, $\sigma(\mathbb A) \cap S^1=\emptyset$, where $S^1$ denotes the unit circle in $\mathbb C$.
Then, using Propositions~\ref{p2x} and~\ref{p3x}, we immediately obtain the conclusion of the theorem.
\end{proof}


\section{Nonautonomous smooth linearization}
\label{secmain}

For a differentiable map $f\colon \mathbb R^d \to \mathbb R^d$, we use $Df(x)$
to denote the derivative of $f$
at a point $x\in \mathbb R^d$. The following is a main result of our paper.

\begin{thm}\label{maintheorem}
Assume that a sequence $(A_m)_{m\in \Z}$ of invertible linear operators  on $\mathbb R^d$ admits
a nonuniform strong exponential dichotomy
and that
 \begin{equation}\label{GB-cond}
  \begin{cases}
a_{k+1}/b_k >\max \{b_r, a_1^{-1}\}, \\
b_i/a_i<b_k^{-1}, \ \forall i=1, \ldots, k, \ b_j/a_j < a_{k+1}, \forall j=k+1, \ldots, r,
  \end{cases}
\end{equation}
where numbers $a_i$ and $b_i$ are as in the statement of  Theorem~\ref{ts}.
Furthermore, let $(f_m)_{m\in \Z}$ be a sequence of $C^1$ functions $f_m \colon \mathbb R^d \to \mathbb R^d$ such that
\begin{itemize}
 \item $f_m(0)=0$ and $D f_m(0) =0$
for each $m\in \mathbb Z$;
 \item there exists $B>0$ such that
 \begin{equation}\label{dsn}
\lVert Df_{m-1}(x)-D f_{m-1}(y)\rVert \le Be^{-\epsilon \lvert m\rvert} \lVert x-y\rVert,\ \ \ \forall x, y\in \mathbb R^d \ \text{and} \ \forall m\in \mathbb Z,
\end{equation}
where $\epsilon \ge 0$ is as in~\eqref{*EE};

\item there exists $\eta>0$ such that
\begin{equation}\label{dns2}
 \lVert D f_{m-1}(x)\rVert \le \eta e^{-\epsilon \lvert m\rvert}, \quad \text{$\forall x\in \mathbb R^d$ and $\forall m\in \Z$.}
\end{equation}
\end{itemize}
Then for sufficiently small $\eta>0$
 there exists a sequence  $(h_m)_{m\in \Z}$ of $C^1$ diffeomorphisms on $\mathbb R^d$
such that
\begin{equation}\label{conj_m}
 h_{m+1}\circ (A_m+f_m)=A_m \circ h_m, \quad m\in \Z.
\end{equation}
\label{main}
\end{thm}

\begin{proof}
Define  a map $F\colon Y_\infty \to Y_\infty$ by
\[
 (F(\mathbf x))_m:=A_{m-1}x_{m-1}+f_{m-1}(x_{m-1}), \quad \mathbf x=(x_m)_{m\in \Z} \in Y_\infty.
\]
Let us establish several auxiliary results.

\begin{lm}
 $F$ is well-defined.
\end{lm}

\begin{proof}[Proof of the lemma]
By~\eqref{dsn}, we see that
 \begin{equation}\label{jk}
  \lVert f_{m-1} (x)\rVert \le Be^{-\epsilon \lvert m\rvert} \lVert x\rVert^2, \quad \forall x\in \mathbb R^d, \forall m\in \Z.
 \end{equation}
It follows from \eqref{ln1}, \eqref{ln2} and~\eqref{jk} that
\[
\begin{split}
 \lVert (F(\mathbf x))_m \rVert_m &=\lVert A_{m-1} x_{m-1}+f_{m-1}(x_{m-1})\rVert_m \\
 &\le C\lVert x_{m-1}\rVert_{m-1}+Ce^{\epsilon \lvert m\rvert} \lVert f_{m-1}(x_{m-1})\rVert \\
 &\le C\lVert x_{m-1}\rVert_{m-1}+BCe^{\epsilon \lvert m\rvert}e^{-\epsilon \lvert m\rvert} \lVert x_{m-1} \rVert^2 \\
 &\le C\lVert x_{m-1}\rVert_{m-1}+BC \lVert x_{m-1} \rVert_{m-1}^2,
 \end{split}
\]
for all $m\in \Z$ and all $\mathbf x=(x_m)_{m\in \Z} \in Y_\infty$.
We conclude that
\[
\sup_{m\in \Z} \lVert (F(\mathbf x))_m \rVert_m <\infty,
\]
i.e., $F(\mathbf x)\in Y_\infty$, and the lemma is proved.
\end{proof}

\begin{lm}\label{dern}
The map $F$ is differentiable and
\[
DF(\mathbf x) \xi=(A_{n-1}\xi_{n-1}+C_{n-1}\xi_{n-1})_{n\in \Z}
\]
for each $\mathbf x=(x_n)_{n\in \Z}$ and $\mathbf \xi=(\xi_n)_{n\in \Z} \in Y_\infty$, where $C_{n-1}=d_{x_{n-1}} f_{n-1}$.
\end{lm}

\begin{proof}[Proof of the lemma]
Given $\mathbf x\in Y_\infty$, we define an operator $L\colon Y_{\infty} \to Y_{\infty}$ by
\[
L\mathbf \xi:=(A_{n-1}\xi_{n-1}+C_{n-1}\xi_{n-1})_{n\in \Z}.
\]
It follows from~\eqref{ln1}, \eqref{ln2} and \eqref{dsn} that $L$ is well defined. Moreover,
\[
\begin{split}
& (F(\mathbf x +\mathbf y)-F(\mathbf x)-L\mathbf y)_n \\
&=f_{n-1}(x_{n-1}+y_{n-1})-f_{n-1}(x_{n-1})-C_{n-1}y_{n-1} \\
&=
\int_0^1 D f_{n-1}(x_{n-1}+t y_{n-1})y_{n-1}\, dt-Df_{n-1}(x_{n-1})  y_{n-1}
\\
&=\int_0^1\bigl(D f_{n-1}(x_{n-1}+t y_{n-1})y_{n-1}-Df_{n-1}(x_{n-1})  y_{n-1}\bigr)\, dt.
\end{split}
\]
Using \eqref{ln1} and~\eqref{dsn} again, we obtain
\[
\begin{split}
& \lVert (F(\mathbf x +\mathbf y)-F(\mathbf x)-L\mathbf y)_n \rVert_n \\
&\le \int_0^1 \lVert D f_{n-1}(x_{n-1}+t y_{n-1})y_{n-1}-Df_{n-1}(x_{n-1})  y_{n-1} \rVert_n \, dt \\
&\le Ce^{\epsilon \lvert n\rvert} \int_0^1 \lVert D f_{n-1}(x_{n-1}+t y_{n-1})y_{n-1}-Df_{n-1}(x_{n-1})  y_{n-1}\rVert \, dt \\
&\le BC \lVert y_{n-1}\rVert_{n-1}^2.
\end{split}
\]
Hence,
\[
\lVert F(\mathbf x +\mathbf y)-F(\mathbf x)-L\mathbf y\rVert_{\infty} \le BC \lVert \mathbf y\rVert_{\infty}^2,
\]
which implies that
\[
\lim_{\mathbf y \to \mathbf 0} \frac{\lVert F(\mathbf x +\mathbf y)-F(\mathbf x)-L\mathbf y\rVert_{\infty}}{\lVert \mathbf y\rVert_{\infty}}=0.
\]
This completes the proof of the lemma.
\end{proof}

\begin{lm}\label{C11}
$F\in C^{1, 1}$, i.e.,
 \[
  \sup_{\mathbf x\neq \mathbf y} \frac{\lVert  DF(\mathbf x)-DF(\mathbf y) \rVert}{\lVert \mathbf x-\mathbf y\rVert_\infty}<\infty.
 \]
\end{lm}

\begin{proof}[Proof of the lemma]
First, by Lemma~\ref{dern}, \eqref{dern} and~\eqref{ln1},
for any $\mathbf \xi =(\xi_n)_{n\in \Z} \in Y_\infty$ satisfying that $\lVert \mathbf \xi\rVert_\infty \le 1$,  we have
\[
 \begin{split}
  \lVert DF(\mathbf x) \mathbf \xi-DF(\mathbf y) \mathbf \xi \rVert_\infty &=\sup_{n\in \Z} \lVert Df_{n-1}(x_{n-1}) \xi_{n-1}-Df_{n-1}(y_{n-1}) \xi_{n-1}\rVert_n \\
  &\le \sup_{n\in \Z} ( Ce^{\epsilon \lvert n\rvert} Be^{-\epsilon \lvert n\rvert}\lVert x_{n-1}-y_{n-1} \rVert \cdot \lVert \xi_{n-1} \rVert) \\
  &= BC \sup_{n\in \Z} (\lVert x_{n-1}-y_{n-1} \rVert_{n-1} \cdot \lVert \xi_{n-1} \rVert_{n-1} )\\
  &\le BC \lVert \mathbf x-\mathbf y\rVert_\infty \cdot \lVert \xi\rVert_\infty \\
  &\le BC \lVert \mathbf x-\mathbf y\rVert_\infty.
 \end{split}
\]
Thus,
$
 \lVert  DF(\mathbf x)-DF(\mathbf y) \rVert \le BC \lVert \mathbf x-\mathbf y\rVert_\infty,
$
and therefore
\[
 \sup_{\mathbf x\neq \mathbf y} \frac{\lVert  DF(\mathbf x)-DF(\mathbf y) \rVert}{\lVert \mathbf x-\mathbf y\rVert_\infty}\le BC <\infty.
\]
The proof of the lemma is completed.
\end{proof}

\begin{lm}\label{0540}
$\lVert DF(\mathbf x)-\mathbb A\rVert \le C\eta$ for all $\mathbf x\in Y_\infty$.
\end{lm}

\begin{proof}[Proof of the lemma]
By Lemma~\ref{dern}, \eqref{ln1} and~\eqref{dns2},
 \[
  \begin{split}
   \lVert DF(\mathbf x)\xi -\mathbb A \xi \rVert_\infty &=\sup_{n\in \Z} \lVert D f_{n-1}(x_{n-1})\xi_{n-1} \rVert_n \\
   &\le \sup_{n\in \Z} (Ce^{\epsilon \lvert n\rvert} \lVert D f_{n-1}(x_{n-1})\xi_{n-1} \rVert)\\
   &\le \sup_{n\in \Z} (Ce^{\epsilon \lvert n\rvert}\eta e^{-\epsilon \lvert n\rvert} \lVert \xi_{n-1}\rVert ) \\
   &\le C\eta \sup_{n\in \Z} \lVert \xi_{n-1}\rVert_{n-1}\\
   &=C\eta \lVert \xi\rVert_\infty,
  \end{split}
\]
for every $\xi \in Y_\infty$, which yields the desired conclusion.
\end{proof}

We are now in the position to complete the proof of the theorem. Note that $\mathbf 0:=(0)_{n\in \Z}$ is a hyperbolic fixed point of~$F$.
Indeed, by Lemma~\ref{dern} and the assumption $Df_n(0)=0$,
we have $DF(\mathbf 0) =\mathbb A$ and, as we have already
noted, $\mathbb A$ is hyperbolic, i.e. $\sigma (\mathbb A)\cap S^1=\emptyset$.
By Theorem~\ref{ts}, Lemmas~\ref{dern}, \ref{C11} and \ref{0540}, and the assumptions of the theorem,
we can apply the Global Smooth Linearization Theorem,
given in the Appendix,
to $F\colon Y_\infty \to Y_\infty$, which implies that there exists a $C^1$ diffeomorphism $\Phi \colon Y_\infty \to Y_\infty$
such that
\begin{equation}\label{conj}
 \Phi \circ F=\mathbb A \circ \Phi.
\end{equation}
 For a fixed $m\in \Z$ and $v\in \mathbb R^d$, define $\mathbf v^m=(v_n)_{n\in \Z}$ by $v_m=v$ and $v_n=0$ for $n\neq m$.
Let $h_m(v):=(\Phi(\mathbf v^m))_m$. It follows readily from~\eqref{conj} that~\eqref{conj_m} holds.

Further, we claim that $h_m$ is differentiable for each $m\in \Z$ and
 \[
  D h_m(v) z=(D\Phi(\mathbf v^m)\mathbf z^m)_m, \quad v, z \in \mathbb R^d.
 \]
In fact,
\begin{align*}
   & \frac{\lVert h_m(v+y)-h_m(v)-(D\Phi(\mathbf v^m)\mathbf y^m)_m\rVert}{\lVert y\rVert} \displaybreak[0]
\\
   &\le Ce^{\epsilon \lvert m\rvert} \cdot \frac{\lVert h_m(v+y)-h_m(v)
-(D\Phi(\mathbf v^m)\mathbf y^m)_m\rVert_m}{ \lVert y\rVert_m} \displaybreak[0]
\\
   &\le  Ce^{\epsilon \lvert m\rvert} \cdot \frac{\lVert \Phi(\mathbf v^m+\mathbf y^m)-\Phi(\mathbf v^m)
-D\Phi(\mathbf v^m)\mathbf y^m \rVert_\infty}{\lVert \mathbf y^m\rVert_\infty}.
\end{align*}
Letting $\lVert y\rVert \to 0$, we have $\lVert \mathbf y^m\rVert_\infty \to 0$ and the desired claim follows.

Then we continue to establish the smoothness of $h_m$.
Note that
\begin{equation}\label{0616}
 \begin{split}
  \lVert Dh_m(v_1)-D h_m(v_2)\rVert &=\sup_{\lVert z\rVert \le 1} \lVert Dh_m(v_1)z-D h_m(v_2)z\rVert\\
  &= \sup_{\lVert z\rVert \le 1} \lVert (D\Phi(\mathbf v_1^m)\mathbf z^m)_m-(D\Phi(\mathbf v_2^m)\mathbf z^m)_m\rVert \\
  &\le \lVert D\Phi(\mathbf v_1^m) -D\Phi(\mathbf v_2^m)\rVert \cdot \sup_{\lVert z\rVert \le 1} \lVert \mathbf z^m \rVert_\infty \\
  &\le Ce^{\epsilon \lvert m\rvert} \lVert D\Phi(\mathbf v_1^m) -D\Phi(\mathbf v_2^m)\rVert.
 \end{split}
\end{equation}
Therefore, if $v_2\to v_1$ in $\mathbb R^d$, then $\mathbf v_2^m\to \mathbf v_1^m$ in $Y_\infty$ and thus (by Lemma~\ref{C11}),
$D\Phi(\mathbf v_2^m)\to  D\Phi(\mathbf v_1^m)$.
Hence, \eqref{0616} implies that $D h_m (v_2)\to Dh_m(v_1)$. We conclude that
$h_m$ is $C^1$ for each $m\in \Z$.

Furthermore, since $\Phi$ is a $C^1$ diffeomorphism,  we have that
\[
 h_m^{-1}(v)=(\Phi^{-1}(\mathbf v^m))_m \quad \text{for $v\in \mathbb R^d$ and $m\in \Z$.}
\]
Hence, one can repeat the above arguments and show that $h_m^{-1}$  is $C^1$ for each $m\in \Z$.
We conclude that $h_m$ is a $C^1$ diffeomorphism on $\mathbb R^d$ for every $m\in \Z$.

Thus, the proof of the theorem is completed.
\end{proof}

Remark that our result of $C^1$ linearization, Theorem~\ref{main}, is obtained in the sense of {\it nonuniform} dichotomies.
The nonuniformity, depending on the initial time in the nonautonomous system, was not considered in \cite{R-S-JDE04,ZZJ}.
A known result (\cite[Section 3]{BV1}) on linearization with such a nonuniformity
is concerning a conjugation with a H\"older continuity.

In the statement of Theorem~\ref{maintheorem},
we assume that the sequence $(A_m)_{m\in \Z}$ admits a nonuniform {\it strong} exponential dichotomy.
We emphasize that
this assumption was crucial for our arguments.
Indeed, if we were to assume that the sequence $(A_m)_{m\in \Z}$ admits a nonuniform exponential dichotomy
which fails to be strong (i.e. only the first and the third inequality in~(5) hold),
then the corresponding sequence $(\lVert \cdot \rVert_m)_{m\in \Z}$
of adapted norms,
which transforms the nonuniform behaviour into the uniform one,
would not necessarily satisfy~\eqref{ln2} (see~\cite{BDV} for a detailed discussion).
Hence, the operator $\mathbb  A$ could fail to be a bounded operator acting
on the space $Y_\infty$ and this would obviously break all of our arguments.
On the other hand,
for higher regularity of conjugation,
those known results in \cite{BDV1}, \cite[Chapter 7]{BVbook} and~\cite[Section 4]{BV1}
all require the sequence $(A_m)_{m\in \Z}$ to admit a nonuniform strong exponential dichotomy.

Our Theorem~\ref{maintheorem} gives a result of global linearization.
This global result is based on the known results (\cite{BDV1}) on global $C^0$ and global H\"older continuous linearization.
Similarly to \cite{ZhangZhang11JFA, ZhangZhang14JDE, ZZJ},
we can also obtain a result of local linearization with weaker assumptions on $(f_m)$.




\section*{Appendix: Global smooth linearization}

In the proof of Theorem~\ref{main} we need a result on global smooth linearization, which can be extended from
the local $C^1$ linearization theorem given in \cite{ZZJ}.

Let $(X,\|\cdot\|)$ be a Banach space and let
$F:X\to X$ be a $C^{1,1}$ diffeomorphism fixing the origin ${\bf 0}$ and
let $\mathbb{A}:=DF({\bf 0})$.
Recall that $F$ can be $C^1$ linearized if
the functional equation (\ref{conj}) has a solution $\Phi$ which is a $C^1$ diffeomorphism.
Moreover, assume that $F$ satisfies
\begin{align}
\|DF(x)-\mathbb{A}\|\le \eta, \quad \forall x\in X,
\label{FFbumpR2}
\end{align}
where $\eta>0$ is a sufficiently small constant, and that the spectrum $\sigma(\mathbb{A})$ satisfy (\ref{spec-cond}).
Then, by the Spectral Decomposition Theory
(see, e.g., \cite[p.9]{GGKbook-90}) one can further assume that the space $X$ has a direct sum decomposition
$
X=X_-\oplus X_+
$
with $\mathbb{A}$-invariant subspaces $X_-$ and $X_+$, that is,
\begin{align}
\mathbb{A}={\rm diag}(\mathbb{A}_-, \mathbb{A}_+),
\label{def-FB}
\end{align}
where $\mathbb{A}_-:X_-\to X_-$ and $\mathbb{A}_+:X_+\to X_+$
are both bounded linear operators such that
\begin{align*}
 \sigma (\mathbb A_-)&=\sigma_-:=\bigcup_{i=1}^k \{ z\in \mathbb C: a_i \le \lvert z\rvert \le b_i <1\},
 \\
 \quad  \sigma (\mathbb A_+)&=\sigma_+:=\bigcup_{j=k+1}^r \{ z\in \mathbb C: 1<a_j \le \lvert z\rvert \le b_j \}.
\end{align*}
We have the following result.

{\bf Global Smooth Linearization Theorem.}
{\it
Let $F$ and $\mathbb{A}$ be given above and assume that the numbers $a_i$ and $b_i$ given in {\rm (\ref{spec-cond})} satisfy {\rm (\ref{GB-cond})}.
Then there exists a $C^1$ diffeomorphism $\Phi: X\to X$ such that equation
{\rm (\ref{conj})} holds, i.e., $F$ can be $C^1$ linearized in $X$.
}

{\bf Proof}.
First of all, we give some notations. Let $C^0_b(\Omega,Z_2)$ consist of all $C^0$ maps $h$ from $\Omega$,
an open subset of a Banach space $(Z_1,\|\cdot\|)$, into
another Banach space $(Z_2,\|\cdot\|)$ such that $\sup_{z\in \Omega}\|h(z)\|<\infty$.
Clearly,  $C^0_b(\Omega,Z_2)$ is a Banach space equipped with the supremum norm $\|\cdot\|_{C^0_b(\Omega,Z_2)}$ defined by
\begin{align*}
\|h\|_{C^0_b(\Omega,Z_2)}:=\sup_{z\in \Omega}\|h(z)\|, \quad \forall h\in C^0_b(\Omega,Z_2).
\end{align*}
For a constant $\gamma>0$, let $S_\gamma(\Omega,Z_2)$ consist of all sequences $u:=(u_n)_{n\ge 0}\subset
C^0_b(\Omega,Z_2)$ such that
$
\sup_{n\ge 0}\{\gamma^{-n}\|u_n\|_{C^0_b(\Omega,Z_2)}\}<\infty.
$
Then, $S_\gamma(\Omega,Z_2)$ is a Banach space equipped with the wighted norm $\|\cdot\|_{S_{\gamma}(\Omega,Z_2)}$
defined by
\begin{align*}
\|u\|_{S_{\gamma}(\Omega,Z_2)}:=\sup_{n\ge 0}\{\gamma^{-n}\|u_n\|_{C^0_b(\Omega,Z_2)}\},
\quad \forall u\in S_{\gamma}(\Omega,Z_2).
\end{align*}

Let $f:=F-\mathbb{A}$ be the nonlinear term of $F$ and let $\pi_-$ and
$\pi_+$ be projections onto $X_-$ and $X_+$ respectively.

Our strategy is firstly to decouple $F$ into a contraction and an expansion by straightening up the invariant foliations.
In order to construct the (stable) invariant foliation, we need to study the Lyapunov-Perron equation (cf. \cite{ChHaTan-JDE97})
\begin{align}
&q_n(x,y_-)
\nonumber\\
&=\mathbb{A}_-^n(y_--\pi_- x)+\sum_{m=0}^{n-1}\mathbb{A}_-^{n-m-1}\{\pi_- f(q_m(x,y_-)+F^m(x))-\pi_- f(F^m(x))\}
\nonumber\\
&-\sum_{m=n}^{\infty}\mathbb{A}_+^{n-m-1}\{\pi_+ f(q_m(x,y_-)+F^m(x))-\pi_+ f(F^m(x))\}, ~~~ \forall n\ge 0,
\label{eqns-foli}
\end{align}
where $q_n:X\times X_-\to X$ is unknown for every integer $n\ge 0$.
For our
purpose of $C^1$ linearization, we need to
find a $C^1$ solution $(q_n)_{n\ge 0}$ of equation (\ref{eqns-foli}),
i.e., each $q_n: X\times X_-\to X$ is $C^1$.

\begin{lm}
Let $F$ and $\mathbb{A}$ be given at the beginning of this section. Assume
that the numbers $a_{k+1}$, $b_k$ and $b_r$ given in {\rm (\ref{spec-cond})} satisfy
\begin{align}
b_kb_r<a_{k+1}.
\label{NR22}
\end{align}
Then, for every neighborhood
$
\Omega_d\subset \{(x,x_-)\in X\times X_-:\|(x,x_-)\|< d\}
$
of ${\bf 0}$ with a given constant $d> 0$, equation {\rm (\ref{eqns-foli})} has a
unique solution
$$
Q_d:=(q_n)_{n\ge 0}\in S_{\gamma_1}(\Omega_d,X)
$$
 such that every $q_n:\Omega_d\to X$ {\rm ($n\ge 0$)} is of class $C^1$,
where $\gamma_1$ is a positive constant satisfying
$
b_k<\gamma_1<1.
$
\label{lm-LPeqn}
\end{lm}

We leave the proof after we finish the proof of the theorem.
Remind that for every $d>0$, we have obtained a solution $Q_d:=(q_n)_{n\ge 0}\in S_{\gamma_1}(\Omega_d,X)$ of equation (\ref{eqns-foli}).
On the other hand, by \cite[Theorem 2.1]{ChHaTan-JDE97} we know that, for every point $(x,y_-)\in X\times X_-$,
equation (\ref{eqns-foli}) has a unique solution $\tilde{Q}(x,y_-):=(\tilde{q}_n(x,y_-))_{n\ge 0}\subset X$ such that
\begin{align*}
\sup_{n\ge 0}\big\{\gamma_1^{-n}\|\tilde{q}_n(x,y_-)\|\big\}<\infty.
\end{align*}
By the uniqueness of $(\tilde{q}_n(x,y_-))_{n\ge 0}$ and the fact
that $(q_n)_{n\ge 0}\in S_{\gamma_1}(\Omega_d,X)$, we have
$
\tilde{Q}|_{\Omega_d}=Q_d.
$
It means that $\tilde{Q}$ is a global $C^1$ solution of equation (\ref{eqns-foli}).
Hence the global (stable) invariant foliation can be constructed by
\begin{align*}
\mathcal{M}_s(x):=\{x+\tilde{q}_0(x,y_-):y_-\in X_-\}, ~~~~~ \forall x\in X.
\end{align*}
The unstable invariant foliation can be obtained by considering the inverse of $F$ under the condition that
\[
a_1a_{k+1}>b_k.
\]
Therefore, by \cite[Theorem 3.1]{Tan-JDE00},
there exists a homeomorphism $\Psi:X\to X$, which and its inverse $\Psi^{-1}:X\to X$ are both $C^{1}$ such that
\begin{align*}
\Psi\circ F&=F_-\circ \pi_-\Psi+F_+\circ \pi_+\Psi,
\end{align*}
where $F_-: X_-\to X_-$ and $F_+:X_+\to X_+$ are both $C^{1,1}$ diffeomorphisms such that
$DF_-({\bf 0})=\mathbb{A}_-$ and $DF_+({\bf 0})=\mathbb{A}_+$.
Recall that $\mathbb{A}_-$ and $\mathbb{A}_+$ are given in (\ref{def-FB}) and have
the spectra $\sigma(\mathbb{A}_-)=\sigma_-$ and
$\sigma(\mathbb{A}_+)=\sigma_+$ respectively.
Then we have the following
result.

\begin{lm}
Let $F_-$ and $F_+$ be given above. Assume
that the numbers $a_i$ and $b_i$ given in {\rm (\ref{spec-cond})} satisfy
\begin{align*}
b_i/a_i<b_k^{-1}, \quad \forall i=1, \ldots, k, \quad b_j/a_j < a_{k+1}, \quad \forall j=k+1, \ldots, r.
\end{align*}
Then there exist $C^1$ diffeomorphisms $\psi_-: X\to X$ and $\psi_+: X\to X$ that linearize $F_-$ and $F_+$ respectively.
\label{lm-holder-B}
\end{lm}

Having found $\psi_-$ and $\psi_+$
in Lemma \ref{lm-holder-B}, we
finally put
\begin{align*}
\Phi=(\psi_-\circ \pi_-+\psi_+\circ \pi_+)\circ \Psi,\quad
\Phi^{-1}=\Psi^{-1}\circ (\psi_-^{-1}\circ \pi_-+\psi_+^{-1}\circ \pi_+).
\end{align*}
One verifies that $\Phi:X\to X$ is a $C^1$ diffeomorphism that linearizes $F$ and the proof of the theorem is completed. \qquad $\Box$

\begin{proof}[Proof of Lemma~\ref{lm-LPeqn}]
Let
\begin{align*}
\|x\|:=\|x_-\|+\|x_+\|, \quad \|(x,y_-)\|:=\|x\|+\|y_-\|
\end{align*}
for $x=x_-+x_+\in X$ and $y_-\in X_-$. Choose two positive numbers $\gamma_1$ and $\gamma_2$ such that
\begin{align*}
b_k<\gamma_1<1<\gamma_2<a_{k+1} \quad{\rm and}\quad \gamma_1b_r<\gamma_2,
\end{align*}
which is possible because of (\ref{NR22}). Let
$$
E_1:=S_{\gamma_1}(\Omega_d,X)  \ \ \  \mbox{and}\ \ \
E_2:=S_{\gamma_2}(\Omega_d,\mathcal{L}(X\times X_-,X))
$$
for short, where $\mathcal{L}(X\times X_-,X)$ is the set of all bounded linear operators map $X\times X_-$ into $X$.
As mentioned at the beginning of the above proof for the theorem, we understand that $E_1$ and $E_2$ are both Banach spaces equipped the corresponding norms,
denoted by $\|\cdot\|_{E_1}$ and $\|\cdot\|_{E_2}$ respectively.
Define operators $\mathcal{T}: E_1\to E_1$ and $\mathcal{S}: E_1\times E_2\to E_2$ by
\begin{align}
&(\mathcal{T} v)_n(x,y_-):=\mathbb{A}_-^n(y_--\pi_- x)
\nonumber\\
&\qquad+\sum_{k=0}^{n-1}\mathbb{A}_-^{n-k-1}\{\pi_- f(v_k(x,y_-)+F^k(x))-\pi_-
f(F^k(x))\}
\nonumber\\
&\qquad-\sum_{k=n}^{\infty}\mathbb{A}_+^{n-k-1}\{\pi_+ f(v_k(x,y_-)+F^k(x))-\pi_+
f(F^k(x))\}
\label{def-T}
\end{align}
and
\begin{align}
&\mathcal{S}(v,w)_n(x,y_-)
:=\Big({\rm
diag}(0,-\mathbb{A}_-^n),\mathbb{A}_-^n\Big)
\nonumber\\
&\qquad+\sum_{k=0}^{n-1}\mathbb{A}_-^{n-k-1}\{D(\pi_-
f)(v_k(x,y_-)+F^k(x))(w_k(x,y_-)+DF^k(x))
\nonumber\\
&\qquad\hspace{2.8cm}-D(\pi_- f)(F^k(x))DF^k(x)\}
\nonumber\\
&\qquad-\sum_{k=n}^{\infty}\mathbb{A}_+^{n-k-1}\{D(\pi_+
f)(v_k(x,y_-)+F^k(x))(w_k(x,y_-)+DF^k(x))
\nonumber\\
&\qquad\hspace{2.8cm}-D(\pi_+ f)(F^k(x))DF^k(x)\}
\label{def-S}
\end{align}
respectively for all $v:=(v_n)_{n\ge 0}\in E_1$ and all $w:=(w_n)_{n\ge 0}\in
E_2$. We claim the following:
\begin{description}
\item[(A1)]
The operators $\mathcal{T}$ and $\mathcal{S}$ are well defined.
\\
\item[(A2)]
The operator $\mathcal{Q}: E_1\times E_2\to E_1\times E_2$ defined by
\begin{align}
\mathcal{Q}(v,w):=(\mathcal{T}v, \mathcal{S}(v,w)), ~~~~~ \forall (v,w)\in
E_1\times E_2,
\label{def-Q}
\end{align}
has an attracting fixed point $(v_*,w_*)\in E_1\times E_2$, i.e.,
\begin{align}
\lim_{n\to \infty}\mathcal{Q}^n(v,w)=(v_*,w_*),
~~~~~ \forall (v,w)\in E_1\times E_2,
\label{limQ}
\end{align}
where $v_*\in E_1$ is the fixed point of $\mathcal{T}$ and
$w_*\in E_2$ is the fixed point of $\mathcal{S}(v_*,\cdot)$.
\label{lm-fc}
\end{description}

Notice that in \cite{ZZJ} the same claims were proved in neighborhoods having sufficiently small diameters $d>0$.
The only difference in the present version is that we allow the diameter to be arbitrarily large. This causes a little change in the estimates, i.e.,
changing $\|(x,x_-)\|\le 1$ into $\|(x,x_-)\|\le d$. Thus, we obtain the following inequalities from \cite{ZZJ}:
\begin{align*}
&\gamma_1^{-n}\|(\mathcal{T} v)_n(x,y_-)\|\le d+K\eta\|v\|_{E_1},
\\
&\gamma_2^{-n}\|\mathcal{S}(v, w)_n(x,y_-)\|\le 1+K_1\|v\|_{E_1}+K_2\eta\|w\|_{E_2},
\\
&\gamma_1^{-n}\|(\mathcal{T} v)_n(x,y_-)-(\mathcal{T}
\tilde{v})_n(x,y_-)\|\le K\eta\,\|v-\tilde{v}\|_{E_1},
\\
&\gamma_2^{-n}\|\mathcal{S}(v, w)_n(x,y_-)-\mathcal{S}(\tilde{v},
\tilde{w})_n(x,y_-)\|
\\
&\hspace{3.4cm}\le
(K_1\|w\|_{E_2}+K_2)\|v-\tilde{v}\|_{E_1}+K\eta\|w-\tilde{w}\|_{E_2},
\end{align*}
where $K_1,K_2$ and $K$ are positive constants. The first two inequalities indicate that
$\mathcal{T}: E_1\to E_1$ and
$\mathcal{S}:E_1\times E_2\to E_2$ are well defined,
i.e., (A1) holds. The third one means that $\mathcal{T}$ is a contraction
since $\eta>0$ is small and therefore has a fixed point $v_*\in E_1$.
Moreover, setting $v=\tilde{v}=v_*$ in the last inequality, we see that
$\mathcal{S}(v_*,\cdot):E_2\to E_2$ is also a contraction and therefore (A2) is proved by the Fiber Contraction Theorem (see e.g. \cite{HirsPugh-70} or \cite[p.111]{Chicone-book99}).

Having (A1) and (A2), we choose an initial point
$\tilde{v}:=0\in E_1$, whose derivative satisfies $D\tilde{v}=0\in E_2$. Moreover, it is obvious that each
$(\mathcal{T}\tilde{v})_n$ is
$C^1$ due to (\ref{def-T}).
Then, by the definitions (\ref{def-T})-(\ref{def-Q})
 of $\mathcal{T}, \mathcal{S}$ and $\mathcal{Q}$, one checks that
$
\mathcal{Q}(\tilde{v}, D\tilde{v})=(\mathcal{T}\tilde{v}, D(\mathcal{T}\tilde{v})),
$
where $D(\mathcal{T}\tilde{v}):=(D(\mathcal{T}\tilde{v}_n))_{n\ge 0}$.
This enables us to prove inductively that
\begin{align}
\mathcal{Q}^n(\tilde{v}, D\tilde{v})=(\mathcal{T}^n\tilde{v}, D(\mathcal{T}^n\tilde{v})), \quad \forall n\ge 0.
\label{Qnwv}
\end{align}
Combining (\ref{limQ}) with (\ref{Qnwv}) we get
$\lim_{n\to\infty}\mathcal{T}^n\tilde{v}=v_*$ and $\lim_{n\to\infty}D(\mathcal{T}^n\tilde{v})=w_*$,
which implies that $v_*\in E_1$ such that
$dv_*=w_*\in E_2$. Since $v_*$ is the fixed point of $\mathcal{T}$, it
is a solution of the Lyapunov-Perron equation (\ref{eqns-foli}).
Thus, the lemma is proved.
\end{proof}

\begin{proof}[Proof of Lemma~\ref{lm-holder-B}]
According to \cite[Lemma 10]{ZZJ} we know that such $\psi_-$ and $\psi_+$ exist in a small neighborhood $U$ of ${\bf 0}$.
We now extend them into the whole space $X$. In what follows, we only consider the expansion $F_+$
since the case of contraction $F_-$ can be solved by considering its inverse.
Choose a sphere $U_0\in U$ such that
$
U_0\subset {\rm int}\, F_+(U_0)\subset U,
$
where ${\rm int}$ denotes the interior of the set $F_+(U_0)$, and define
$$
X_i:=F_+^{i+1}(U_0)\backslash F_+^{i}(U_0),\quad V_0:=\psi_+(U_0),\quad Y_i:=\mathbb{A}_+^{i+1} (V_0)\backslash \mathbb{A}_+^i(V_0)
$$
for all $i\in \mathbb{N}\cup\{0\}$.
It is clear that
\begin{align*}
X_i\cap X_j=\emptyset, ~~ \forall i\ne j, ~~ X_i\cap U_0=\emptyset, \quad &U_0\cup \bigcup_{i=0}^{\infty}X_i=X,~ ~F_+(X_i)=X_{i+1},
\\
Y_i\cap Y_j=\emptyset,  ~~\forall i\ne j,   ~~Y_i\cap V_0=\emptyset,  \quad &V_0\cup \bigcup_{i=0}^{\infty}Y_i=X,   ~~\mathbb{A}_+(Y_i)=Y_{i+1}.
\end{align*}
Then we define
$$
\psi_0:=\psi_+|_{X_0},\qquad \psi_i:=\mathbb{A}_+\circ \psi_{i-1}\circ F_+^{-1},\quad\forall i\in \mathbb{N},
$$
and define the global solution $\psi_*$ by
$$
\psi_*(x):=
\left\{\begin{array}{lll}
\psi_+(x), & \quad\forall x\in U_0,
\\
\psi_i(x), & \quad\forall x\in X_i, ~~\forall i\in \mathbb{N},
\end{array}\right.
$$
which is $C^1$ in $X$. Moreover, one verifies that $\psi_*(X_i)=Y_i$ for all $i\in \mathbb{N}\cup\{0\}$ because
$$
\psi_*\circ F_+^{i+1}(U_0)= \mathbb{A}_+^{i+1}\circ \psi_*(U_0)= \mathbb{A}_+^{i+1}(V_0),
\quad
\psi_*\circ F_+^{i}(U_0)= \mathbb{A}_+^{i}(V_0).
$$
Thus, $\psi_*:X\to X$ is one-to-one
and therefore it is a $C^1$ diffeomorphism that linearizes $F_+$. Without loss of generality, 
we still use $\psi_+$ to denote $\psi_*$ and
the proof is completed.
\end{proof}

{\bf Acknowledgment:} The authors are ranked in alphabetic order.
The author Davor Dragi\v{c}evi\'c is supported by
an Australian Research Council Discovery Project DP150100017, Croatian Science Foundation under the project IP-2014-09-2285 and by
the University of Rijeka research grant 13.14.1.2.02.
 The author Weinian Zhang
is supported by NSFC grants \#11771307 and \#11521061. The author Wenmeng Zhang is supported by NSFC grant \#11671061.


\bibliographystyle{amsplain}

\enlargethispage{1cm}

\end{document}